\documentclass[12pt]{article}

\usepackage{amsmath,amssymb,enumerate,bbm}
\newtheorem{theorem}{Theorem}[section]

\newtheorem{remark}[theorem]{Remark}

\numberwithin{equation}{section}
\newenvironment{proof}{\noindent\textbf{Proof\ }}{\hspace*{\fill}$\Box$\medskip}

 \makeatletter
 
 \newcommand{\Rmnum}[1]{\expandafter\@slowromancap\romannumeral #1@}
 \makeatother

\begin{document}

\title{On the sum of the squared multiplicities of the distances in a point set over finite fields}
\author{\begin{tabular}[t]{cc} 
  Le Anh Vinh & Dang Phuong Dung\\
  Mathematics Department      & International Business School \\
  Harvard University     & Brandeis University \\
  \textit{vinh@math.harvard.edu}         & \textit{lizdang@brandeis.edu}
\end{tabular}}
\maketitle

\begin{abstract}
 We study a finite analog of a conjecture of Erd\"os on the sum of the squared multiplicities of the distances determined by an $n$-element point set. Our result is based on an estimate of the number of hinges in spectral graphs.
\end{abstract}

\begin{center}
\end{center}

\section{Introduction}

Let $\mathbbm{\mathbbm{F}}_q$ denote the finite field with $q$ elements where
$q \gg 1$ is an odd prime power. For any $x,y \in \mathbbm{F}_q^d$, the distant between $x,y$ is defined as $||x-y|| = (x_1 - y_1)^2 + \ldots + (x_d - y_d)^2$. Let $E \subset \mathbbm{F}_q^d$, $d \geqslant
2$. Then the finite analog of the classical Erd\"os distance problem is to determine
the smallest possible cardinality of the set
\[ \Delta (E) =\{||x-y|| : x, y \in
   E\}, \]
viewed as a subset of $\mathbbm{F}_q$. Bourgain, Katz and Tao (\cite{bourgain-katz-tao}) showed, using
intricate incidence geometry, that for every $\varepsilon > 0$, there exists
$\delta > 0$, such that if $E \in \mathbbm{F}_q^2$ and $|E| \leqslant C_{\epsilon} q^{2 -
\epsilon}$, then $ |\Delta (E)| \geqslant C_{\delta} q^{\frac{1}{2} +
\delta}$ for some constants $C_{\epsilon}, C_{\delta}$. The relationship
between $\varepsilon$ and $\delta$ in their argument is difficult to
determine. Going up to higher dimension using arguments of Bourgain, Katz and
Tao is quite subtle. Iosevich and Rudnev (\cite{iosevich-rudnev}) establish the following results
using Fourier analytic methods.

\begin{theorem}\label{ir1} (\cite{iosevich-rudnev})
  Let $E \subset \mathbbm{F}_q^d$ such that $|E| \gtrsim C q^{d / 2}$ for $C$
  sufficient large. Then
  \begin{equation}
    |\Delta (E)| \gtrsim \min \left\{ q, \frac{|E|}{q^{\frac{d - 1}{2}}}
    \right\} .
  \end{equation}
\end{theorem}

In \cite{vinh-ejc1}, the author gives another proof of this result using the graph theoretic method. This method also works for many other related problems, see \cite{vinh-ejc2,vinh-ejc,vinh-fkw,vinh-dg}. The advantages of the graph theoretic method are twofold. First, we can reprove and sometimes improve several known results in vector spaces over finite fields. Second, our approach works transparently in the non-Euclidean setting. In this note, we use the same method to study a finite analog of a related conjecture of Erd\"os.  

Let $\deg_S(p,r)$ denote the number of points in $S$ at distance $r$ from $p$. A conjecture of of Erd\"os \cite{erdos90} on the sum of the squared multiplicities of the distances determined by an $n$-element point set states that 
\[ \sum_{r>0} \left(\sum_{p\in S} \deg_S(p,r)^2 \right) \leq O(n^3(\log n)^{\alpha},\]
for some $\alpha > 0$. For this function, Akutsu et al. \cite{akutsu98} obtained the upper bound $O(n^{3.2})$, improving an earlier result of Thiele \cite{thiele95}. If no three points are collinear, Thiele gives the better bound $O(n^3)$. This bound is sharp by the regular $n$-gons \cite{thiele95}. Nothing is known about this function over higher dimensional spaces.
The purpose of this note is to study this function in the space $\mathbbm{F}_q^d$. To avoid some null distance pairs (i.e. two distinct points with distance zero), we assume that $-1$ is not a square in $\mathbbm{F}$ throughout this note. The main result of this note is the following.

\begin{theorem}\label{mt-ssm}
Let $E \subset \mathbbm{F}_q^d$. For any point $p \in E$ and a distant $r \in \mathbbm{F}_q^{*}$, Let $\deg_E(p,r)$ denotes the number of points in $E$ at distance $r$ from $p$. Let $f(E)$ denote the sum of the square multiplicities of the distances determined by $E$:
\[f(E) = \sum_{r \in \mathbbm{F}_q^{*}} \left(\sum_{p\in E} \deg_E(p,r)^2\right).\]

a) Suppose that $|E| \geq \Omega(q^{\frac{d+1}{2}})$ then $f(E) = \Theta(|E|^3/q)$.

b) Suppose that $|E| \leq O(q^{\frac{d+1}{2}})$ then $\Omega(|E|^3/q) \leq f(E) \leq O(|E|q^d)$. 

\end{theorem}

The rest of this note is organized as follows. In Section 2, we establish an estimate about the number of hinges (i.e. ordered paths of length $2$) in spectral graphs. Using this estimate, we give a proof of Theorem \ref{mt-ssm} in Section 3.  

\section{Number of hinges in an $(n, d, \lambda)$-graph}

We call a graph $G = (V, E)$ $(n, d, \lambda)$-graph if $G$ is a $d$-regular
graph on $n$ vertices with the absolute values of each of its eigenvalues but
the largest one is at most $\lambda$. It is well-known that if $\lambda \ll d$
then an $(n, d, \lambda)$-graph behaves similarly as a random graph $G_{n, d /
n}$. Precisely, we have the following result (cf. Theorem 9.2.4 in \cite{alon-spencer}).

\begin{theorem} (\cite{alon-spencer}) \label{tool 1}
  Let $G$ be an $(n, d, \lambda)$-graph. For a vertex $v \in V$ and a subset
  $B$ of $V$ denote by $N (v)$ the set of all neighbors of $v$ in $G$, and
  let $N_B (v) = N (v) \cap B$ denote the set of all neighbors of $v$ in $B$.
  Then for every subset $B$ of $V$:
  \begin{equation}
    \sum_{v \in V} (|N_B (v) | - \frac{d}{n} |B|)^2 \leqslant
    \frac{\lambda^2}{n} |B| (n - |B|) .
  \end{equation}
\end{theorem}

The following result is an easy corollary of Theorem \ref{tool 1}

\begin{theorem} \label{tool 2}
  (cf. Corollary 9.2.5 in \cite{alon-spencer}) Let $G$ be an $(n, d, \lambda)$-graph. For every
  set of vertices $B$ and $C$ of $G$, we have
  \begin{equation}
    |e (B, C) - \frac{d}{n} |B\|C\| \leqslant \lambda \sqrt{|B\|C|},
  \end{equation}
  where $e (B, C)$ is the number of edges in the induced bipartite subgraph of
  $G$ on $(B, C)$ (i.e. the number of ordered pair $(u, v)$ where $u \in B$,
  $v \in C$ and $u v$ is an edge of $G$). 
\end{theorem}

From Theorem \ref{tool 1} and Theorem \ref{tool 2}, we can derive the following estimate about the number of hinges in an $(n,d,\alpha)$-graph.

\begin{theorem} \label{path}
Let $G$ be an $(n, d, \lambda)$-graph. For every set of vertices $E$ of $G$, we have
  \begin{equation}
    p_2(E) \leq |E| \left( \frac{d|E|}{n} + \lambda \right)^2,
  \end{equation}
  where $p_2(E)$ is the number of ordered paths of length two in $E$
  (i.e. the number of ordered triple $(u, v, w) \in E \times E \times E$ with
  $u v$, $v w$ are edges of $G$).
\end{theorem}

\begin{proof}
  For a vertex $v \in V$ let $N_E(v)$ denote the set of all
  neighbors of $v$ in $E$. From Theorem \ref{tool 1}, we have
  \begin{equation}
    \sum_{v \in E} (|N_E (v) | - \frac{d}{n} |E|)^2 \leqslant \sum_{v \in V}
    (|N_E (v) | - \frac{d}{n} |E|)^2  \leqslant  \frac{\lambda^2}{n} |E|
    (n - |E|).
  \end{equation}
  This implies that
  \begin{equation} \label{1}
    \sum_{v \in E} N_E^2 (v) + \left( \frac{d}{n} \right)^2
    |E|^3 - 2\frac{d}{n} |E| \sum_{v \in E} N_E(v)
    \leqslant \frac{\lambda^2}{n} |E| (n - |E|)
  \end{equation}
  From Theorem \ref{tool 2}, we have
  \begin{equation}\label{2}
    \sum_{v \in E} N_E (v) \leq \frac{d}{n} |E|^2 + \lambda |E|.
  \end{equation}
  Putting (\ref{1}) and (\ref{2}) together, we have
  \begin{eqnarray*}
    \sum_{v \in E} N^2_E (v) &\leq& \left( \frac{d}{n} \right)^2
    |E|^3 + 2 \frac{\lambda d}{n} |E|^2 + \frac{\lambda^2}{n} |E| (n - |E|)\\
    & < & \left( \frac{d}{n} \right)^2
    |E|^3 + 2 \frac{\lambda d}{n} |E|^2 + \lambda^2|E|\\
    & = & |E| \left( \frac{d|E|}{n} + \lambda \right)^2,
  \end{eqnarray*}
  completing the proof of the theorem.
\end{proof}

\section{Proof of Theorem \ref{mt-ssm}}
Let $\mathbbm{F}_q$ denote the finite field with $q$ elements where $q \gg 1$
is an odd prime power. For a fixed $a \in \mathbbm{F}_q^{\ast}$, the finite Euclidean graph $G_q(a)$ in $\mathbbm{F}_q^d$ is defined as the graph with vertex
set $\mathbbm{F}_q^d$ and the edge set
\[ E =\{(x, y) \in \mathbbm{F}_q^d \times \mathbbm{F}_q^d \mid x \neq y, ||x - y|| = a\}, \]
where $||.||$ is the analogue of Euclidean distance $||x|| = x_1^2+\ldots+x_d^2$.
In \cite{medrano}, Medrano et al. studied the spectrum of these graphs and showed that these graphs are asymptotically Ramanujan graphs.  They proved the following result.

\begin{theorem} (\cite{medrano}) \label{tool 4}
The finite Euclidean graph $G_q(a)$ is regular of valency $(1+o(1))q^{d-1}$ for any $a \in \mathbbm{F}_q^{\ast}$. Let $\lambda$ be any eigenvalues of the graph $G_q(a)$ with $\lambda \neq$ valency of the graph then 
\begin{equation}
|\lambda| \leq 2q^{\frac{d-1}{2}}.
\end{equation}
\end{theorem}

We have the number of ordered triple $(u,v,w) \in E \times E \times E$ with $uv$ and $vw$ are edges of $G_q(a)$ is $\sum_{p \in E} deg_E(p,a)^2$. From Theorem \ref{path} and Theorem \ref{tool 4}, we have
\begin{equation}\label{upper}
f(E) \leq \sum_{a \in \mathbbm{F}_q^{*}} |E|\left((1+o(1))\frac{|E|}{q} + 2q^{\frac{d-1}{2}}\right)^2 \leq (q-1)|E|\left((1+o(1))\frac{|E|}{q} + 2q^{\frac{d-1}{2}}\right)^2.
\end{equation}
Thus, if $|E| \geq \Omega(q^{\frac{d+1}{2}})$ then
\begin{equation}\label{up1}
f(E) \leq O(|E|^3/q),
\end{equation} 
and if $|E| \ll O(q^{\frac{d+1}{2}})$ then
\begin{equation}\label{up2}
f(E) \leq O(|E|q^d).
\end{equation}
We now give a lower bound for $f(E)$. We have
\begin{eqnarray}
  f (E) & = & \sum_{r \in \mathbbm{F}^{\ast}_q} \left( \sum_{p \in E} \deg_E
  (p, r)^2 \right) \nonumber\\
  & \geqslant & \sum_{r \in \mathbbm{F}_q^{\ast}} \frac{1}{|E|} \left(
  \sum_{p \in E} \deg_E (p, r) \right)^2 \nonumber\\
  & \geqslant & \frac{1}{(q - 1) |E|} \left( \sum_{r \in
  \mathbbm{F}_q^{\ast}} \sum_{p \in E} \deg_E (p, r) \right)^2 \nonumber\\
  & \geqslant & \frac{|E| (|E| - 1)^2}{(q - 1)} = \Omega (|E|^3 / q) .
  \label{lower}
\end{eqnarray}

Theorem \ref{mt-ssm} follows immediately from (\ref{up1}), (\ref{up2}) and (\ref{lower}).

\begin{remark} From the above proof, we can derive Theorem \ref{ir1} as follows.
\[
\frac{1}{|\Delta(E)||E|}\left( |E|(|E|-1) \right)^2 
  \leq f(E) \leq |\Delta(E)||E|\left((1+o(1))\frac{|E|}{q} + 2q^{\frac{d-1}{2}}\right)^2.
\]
This implies that
\[
|\Delta(E)| \geq \frac{(1+o(1))q}{1+2\frac{q^{(d+1)/2}}{|E|}},
\]
and Theorem \ref{ir1} follows immediately.
\end{remark}

\end{document}